\documentclass[11pt]{amsart}
\usepackage{geometry}                
\usepackage{graphicx}
\usepackage{amssymb}
\usepackage{epstopdf}
\usepackage{amsmath}
\usepackage{amsfonts}
\usepackage{amssymb}
\usepackage{amsthm}

\usepackage[parfill]{parskip}

\usepackage{mathrsfs}
\usepackage{latexsym}
\usepackage{longtable}    
\usepackage{graphpap}     
\DeclareMathAlphabet{\mathpzc}{OT1}{pzc}{m}{it}

\newtheorem{thm}{Theorem}[section]

\newtheorem{cor}[thm]{Corollary}

\theoremstyle{definition}
\newtheorem{defin}[thm]{Definition}
\newtheorem{ex}[thm]{Example}
\theoremstyle{remark}
\newtheorem{rem}[thm]{Remark}
\numberwithin{equation}{section}
\numberwithin{thm}{section}

\def \R {\mathbb R}
\def \Z {\mathbb Z}
\def\Q {\mathbb Q} 
\def\N {\mathbb N}

\newcommand{\kf}{\mathbb{K}}
\newcommand{\D}{\mathpzc{D}}

\newcommand{\lid}{\lambda I - D}
\newcommand{\eo}{E_{\omega}}
\newcommand{\se}{\sigma_{e}}
\newcommand{\spp}{\sigma_{p}}
\newcommand{\sep}{\sigma_{e}^{\prime}}
\newcommand{\sepp}{\sigma_{e}^{\prime\prime}}

\title[Spectral Analysis for a class of Unbounded Linear Operators]
{Spectral and Essential Spectral Analysis of Finite-Rank Perturbations of Unbounded Diagonal Operators on Non-Archimedean Hilbert Spaces}
\author{Teylama Herve Miabey}

\begin{document}
\maketitle
\begin{abstract}
\thispagestyle{plain}
\setcounter{page}{1}       
\addcontentsline{toc}{chapter}{abstract}
We study the spectral properties of a class of unbounded linear operators on a non-Archimedean Hilbert space $E_{\omega}$. More precisely, we consider operators of the form
\[
T=D+F,\qquad 
F=\sum_{j=1}^{m} u_j\otimes v_j,
\]
where $D$ is an unbounded diagonal operator and $F$ is a finite-rank perturbation. This work extends the spectral analysis of Diagana and McNeal for rank-one perturbations of diagonal operators to the case of arbitrary finite-rank perturbations. The main objective is to describe the spectrum, point spectrum, and essential spectrum of such operators in terms of the diagonal sequence associated with $D$ and the Fredholm properties of $\lambda I-T$. The theory of Fredholm operators plays a central role, particularly in the computation of the essential spectrum and in the study of stability under finite-rank perturbations.
\end{abstract}
\section{Introduction and Background}

Spectral theory over non-Archimedean fields has developed as an important counterpart to the classical spectral theory of operators on complex Banach and Hilbert spaces. In the non-Archimedean setting, the underlying scalar field is equipped with an absolute value satisfying the strong triangle inequality, and this changes many of the geometric and analytic properties of the corresponding normed spaces. The present paper is concerned with the spectral analysis of a class of unbounded diagonal operators subject to finite-rank perturbations in a non-Archimedean Hilbert space.

Let $\kf$ be a field. An absolute value on $\kf$ is a function
\[
|\cdot|:\kf\longrightarrow [0,\infty)
\]
such that, for all $x,y\in \kf$,
\[
|x|=0 \Longleftrightarrow x=0,\qquad |xy|=|x||y|,
\]
and
\[
|x+y|\leq |x|+|y|.
\]
The absolute value is called non-Archimedean if it satisfies the stronger inequality
\[
|x+y|\leq \max\{|x|,|y|\}.
\]
A field $\kf$ endowed with such an absolute value is called a non-Archimedean valued field. Throughout this work, the scalar field is assumed to be a complete non-Archimedean valued field whenever completeness is needed.

The non-Archimedean inequality implies several geometric properties that differ sharply from the classical Archimedean case. For example, if $|x|\neq |y|$, then
\[
|x+y|=\max\{|x|,|y|\}.
\]
Consequently, in a non-Archimedean metric space, triangles are isosceles in the sense that among the three distances determined by any three points, the maximum is attained at least twice. If the metric is induced by the absolute value,
\[
d(x,y)=|x-y|,
\]
then open and closed balls have unusual properties: every point of a ball is a center of the ball, balls are both open and closed, and two balls are either disjoint or one contains the other. These facts are fundamental in non-Archimedean analysis and are used repeatedly in the study of convergence, continuity, and boundedness.

A sequence $(x_n)$ in a non-Archimedean valued field is Cauchy if and only if
\[
|x_{n+1}-x_n|\longrightarrow 0
\qquad \text{as } n\to\infty.
\]
In particular, if $\kf$ is complete, then a series $\sum_{n=0}^{\infty}x_n$ converges if and only if $x_n\to 0$. This criterion is one of the most useful consequences of the ultrametric inequality and is frequently used in the construction and analysis of non-Archimedean Banach spaces.

A basic example of a non-Archimedean field is the field of $p$-adic numbers. Let $p$ be a prime number. The $p$-adic valuation $V_p$ on $\Q$ is defined by writing every nonzero rational number $x$ in the form
\[
x=p^{V_p(x)}\frac{a}{b},
\]
where $a,b\in \Z$ are relatively prime to $p$. The associated $p$-adic absolute value is
\[
|x|_p=p^{-V_p(x)},\qquad x\neq 0,
\]
and $|0|_p=0$. This absolute value satisfies
\[
|x+y|_p\leq \max\{|x|_p,|y|_p\},
\]
so $(\Q,|\cdot|_p)$ is a non-Archimedean valued field. Since $\Q$ is not complete with respect to $|\cdot|_p$, its completion is denoted by $\Q_p$ and is called the field of $p$-adic numbers. The field $\Q_p$ is a central example motivating the general theory developed in this paper.

More generally, valuations provide a natural way to construct non-Archimedean absolute values. A valuation on a field $\kf$ is a map
\[
V:\kf\longrightarrow \R\cup\{\infty\}
\]
such that
\[
V(x)=\infty \Longleftrightarrow x=0,\qquad
V(xy)=V(x)+V(y),
\]
and
\[
V(x+y)\geq \min\{V(x),V(y)\}.
\]
If $c>1$, then
\[
|x|=c^{-V(x)}
\]
defines a non-Archimedean absolute value on $\kf$. A valuation is called discrete if its value group is isomorphic to $\Z$. In the discrete case, one obtains valuation rings and maximal ideals analogous to those appearing in the construction of $\Q_p$.

Finite-dimensional non-Archimedean vector spaces also play an important role. For a non-Archimedean field $\kf$, the vector space $\kf^t$ is defined by
\[
\kf^t=\{x=(x_1,\ldots,x_t):x_i\in \kf\}.
\]
It is equipped with the norm
\[
\|x\|_t=\max_{1\leq i\leq t}|x_i|.
\]
This norm is non-Archimedean. One may also define the bilinear form
\[
\langle x,y\rangle_t=\sum_{i=1}^{t}x_i y_i,
\]
which satisfies the estimate
\[
|\langle x,y\rangle_t|\leq \|x\|_t\|y\|_t.
\]
The canonical vectors $e_1,\ldots,e_t$ form the standard orthonormal basis of $\kf^t$.

The infinite-dimensional setting considered in this paper is a non-Archimedean Hilbert space $E_{\omega}$, typically realized as a weighted sequence space over $\kf$. Its elements are sequences
\[
u=\sum_{i\geq 0} u_i e_i
\]
satisfying an appropriate convergence condition determined by a weight sequence $\omega=(\omega_i)_{i\geq 0}$. The family $(e_i)_{i\geq 0}$ serves as a distinguished orthonormal basis. Such spaces provide a natural framework for studying diagonal operators and their perturbations in the non-Archimedean setting; see, for example, \cite{TOK}, \cite{vR}, and \cite{DM}.

The principal operator studied in this paper is
\[
T=D+F,
\]
where $D$ is a diagonal operator and $F$ is a finite-rank operator of the form
\[
F=\sum_{j=1}^{m}u_j\otimes v_j.
\]
Here the rank-one operator $u\otimes v$ is defined by
\[
(u\otimes v)(x)=\langle x,v\rangle u.
\]
If
\[
D\left(\sum_{i\geq 0}x_i e_i\right)
=
\sum_{i\geq 0}\lambda_i x_i e_i,
\]
then the spectral properties of $D$ are closely related to the set
\[
\Lambda=\{\lambda_i:i\geq 0\}.
\]
When $D$ is unbounded, its domain is a proper dense subspace of $E_{\omega}$, and one must carefully distinguish between boundedness, closedness, invertibility, and Fredholm properties.

For a linear operator $A$ on a non-Archimedean Banach space, the resolvent set $\rho(A)$ consists of all $\lambda\in \kf$ such that
\[
\lambda I-A
\]
is bijective and has bounded inverse. The spectrum is
\[
\sigma(A)=\kf\setminus \rho(A).
\]
The point spectrum $\sigma_p(A)$ is the set of eigenvalues of $A$, namely those $\lambda\in \kf$ for which
\[
N(\lambda I-A)\neq \{0\}.
\]
The essential spectrum is described in terms of Fredholm theory.

Recall that an operator $A:X\to Y$ between non-Archimedean Banach spaces is called Fredholm if its kernel $N(A)$ is finite-dimensional, its range $R(A)$ is closed, and the quotient space $Y/R(A)$ is finite-dimensional. Its Fredholm index is
\[
\operatorname{ind}(A)=\dim N(A)-\operatorname{codim} R(A).
\]
The Fredholm framework is useful because essential spectra are often characterized by the failure of $\lambda I-A$ to be Fredholm, or to be Fredholm of index zero. In particular, for an operator $A$, one may define the essential spectrum by
\[
\sigma_e(A)
=
\{\lambda\in \kf: \lambda I-A \text{ is not Fredholm of index }0\}.
\]
This point of view is especially important for finite-rank perturbations, since Fredholm properties are frequently stable under compact or finite-rank perturbations; see \cite{A}, \cite{PE2}, and \cite{PE3}.

The present work extends the rank-one perturbation results of Diagana and McNeal \cite{DM}, who studied operators of the form
\[
A=D_{\lambda}+u\otimes v
\]
on a non-Archimedean Hilbert space. Their results describe the spectrum in terms of the diagonal sequence and additional eigenvalues generated by the rank-one perturbation. In this paper, we consider the more general finite-rank perturbation
\[
T=D+\sum_{j=1}^{m}u_j\otimes v_j.
\]
Our goal is to determine how the spectral and essential spectral properties of the unbounded diagonal operator $D$ change under such finite-rank perturbations. The main tools are the structure of non-Archimedean Hilbert spaces, the behavior of diagonal operators, and the Fredholm theory of non-Archimedean linear operators.

\section{Computation of the Essential Spectrum of $D$, $\se(D)$ where
$D$ is an unbounded linear operator}
\begin{thm}
Let $A : E \rightarrow F$ be a linear map from $E$ to $F$, where $E$ 
and $F$ are normed vector spaces over $\kf$, with $\kf$ a complete
non-Archimedian field such that $\{|\lambda|,\lambda \in \kf\}$ is
dense in $\R_{+}^{\ast}$. $A$ is continuous off there exists a constant
$M$ such that $||A(x)|| \le M||x||\, \forall x \in E$. 
\end{thm}
\begin{proof}
Suppose there exists such a constant $M \ge 0$ such that 
$||A(x)|| \le M||x||\, \forall x \in E$. Then 
\[
\begin{aligned}
A(x) - A(x^{\prime}) &= A(x - x^{\prime}) \\
\Rightarrow  ||A(x) - A(x^{\prime})|| & = ||A(x - x^{\prime})|| \\
\Rightarrow  ||A(x - x^{\prime})|| &\le M||x - x^{\prime}||.
\end{aligned}
\]
Hence $A$ is continuous.

Suppose now that $A$ is continuous. Let $U = \left\{ y \in F, 
||y|| < 1\right\}$. Then $U$ is open in$F$. Since $A$ is continuous, 
there exists an open set $V \in E$ such that $V = A^{-1}(U)$. If 
$0 \in U$ then $0 \in V$. There exists a positive real number 
$r$ such that $W = \{x \in E\,|\, ||x|| < r\}$ is contained in 
$V$. Let $x \in E$, $x \ne 0$, $\varepsilon > 0, \alpha \in \kf,
0 < |\alpha| < r$ and $\beta \in \kf$ such that $||x|| <
|\beta| < ||x||+\varepsilon$. Then 
$\displaystyle x^{\prime} = \frac{\alpha}{\beta}x \in W$, because
$\displaystyle \frac{|\alpha|}{|\beta|}||x||<|\alpha|<r$.

Since $A(W) \subseteq A(V) \subset U$, we have $||x^{\prime}|| < 1$. 
Since $\displaystyle A(x^{\prime}) = \frac{\alpha}{\beta}A(x)$, then
\[
||A(x)||\le \frac{|\beta|}{|\alpha|}\le \frac{||x||+\varepsilon}{
|\alpha|},
\]
which is true for all $\alpha \in \kf$ such that $0 < |\alpha| < r$. 
Since $\{|\lambda|, \lambda \in \kf \}$ is dense in $\R_{+}$, 
$\displaystyle ||A(x)|| \le \frac{||x||+ \varepsilon}{r}$. So the 
relation $\displaystyle ||A(x)|| \le \frac{||x||+ \varepsilon}{r}$
is true for every $\varepsilon > 0$. This implies the relation
$||A(x)|| < M||x||$ with $M = 1/r$.
\end{proof}
\begin{rem}
If the valuation on $\kf$ is not discrete, then 
$\{|\lambda|, \lambda \in \kf \}$ is dense in $\R_{+}$. If the 
valuation on $\kf$ is discrete, $\{|\lambda|, \lambda \in \kf \}$ is 
not dense in $\R_{+}$.
\end{rem}

In the case of an unbounded linear operator, we will assume the the
 evaluation of the field $\kf$ is not discrete. 

\begin{defin}
Let $\D(D)\subseteq \eo \rightarrow \eo$ be a diagonal operator, i.e
there exists $\lambda_{i} \in \kf, i \in \N$ such that $D(u) =\sum 
\lambda_{i}u_{i}e_{i}$, $u = \sum u_{i}e_{i}$ with $u \in \D(D)$, 
$D$ is an unbounded linear operator and $\kf$ is a non-Archimedian
value field. We have
\[
\D = 
\left\{u = \sum u_{i}e_{i}, \lim |\lambda_{n}||u_{n}||w_{n}|^{1/2} =0
\right\}
\], where $e_{i}\in  \D(D) \forall i \in \N \Rightarrow 
\overline{\D(D)} = \eo$. 
\end{defin}

\begin{defin}
Let $\rho(D)$ be the resolvent of the unbounded diagonal and linear
operator of $D$. $\lambda \in \kf$ is an element of $\rho(D)$ iff
$\lid$ is a bijection and $(\lid)^{-1}$ is bounded. The spectrum of 
$D$ denoted by $\sigma(D)$ is given by $\kf - \rho(D)$. 
\end{defin}

\begin{thm}
If $\D \ne \eo$, then $\sigma(D) = \kf$. 
\end{thm}

\begin{proof}
Since $\D \ne \eo$, $D$ is not continuous ($D$ is not bounded in
$\D(D))$, otherwise $D$ can be extended to a continuous map on $\eo$. 
Since $\overline{\D(D)} = \eo$, we have $\D(D) = 
\overline{\D(D)} = \eo$, but by Theorem 6, we have that 
$\lambda \in \rho(D)$. $(\lid)^{-1}$ is continuous $\Rightarrow 
\lid$ is continuous $\Rightarrow D$ is continuous. Therefore
$\rho(D) = \phi$, which gives us $\sigma(D) = \kf$. 

$\lambda \in \kf$ is an eigenvalue off $N(\lid) \ne 0$. Given that
$D(u) = \sum \lambda_{i}u_{i}e_{i}$,$ u = \sum u_{i}e_{i}, u \ne 0$,
\[
\begin{aligned}
N(\lid) \ne 0 & \Leftrightarrow (\lid)(u) = 
\sum(\lambda - \lambda_{i})(u_{i})e_{i} = 0\\
& \Rightarrow (\lambda - \lambda_{i})(u_{i}) = 0 \quad \forall i \in \N.
\\
\end{aligned}
\]  
Since $u \ne 0$, $\exists k \in \N$ such that $u_{k}\ne 0$. Hence
$(\lambda - \lambda_{k})u_{k} = 0 \Rightarrow \lambda - \lambda_{k}= 0$.
Therefore $\lambda = \lambda_{k}\in \Lambda \Rightarrow \spp(D) 
\subseteq \Lambda$. 

Note that for all $\lambda_{k}$, $\lambda_{k}I -D)e_{k}
= \lambda_{k}e_{k}-\lambda_{k}e_{k} = 0 \Rightarrow \lambda_{k}
\in \spp(D)$. Then $\spp(D) \supseteq \Lambda$. Therefore $\spp(D) = 
\Lambda$.
Now recall that 

$\sep(D) = \{\lambda \in \kf\,|\,
(\lid)\textrm{ is injective but not surjective}\}$ and

$\sepp(D) = \{\lambda \in \kf\,|\,\lambda \textrm{ is of infinite 
multiplicity with respect to }D\}$. We conclude the following:
\[ 
\begin{aligned}
\sep(D) &= \kf - \Lambda \\
\sepp(D) &\subseteq \Lambda - \Lambda^{\ast}, \textrm{ where }
\Lambda^{\ast} = \{\lambda \in \Lambda : r_{\lambda}< \infty\}\\
\se(D) &= \sep(D) \cup \sepp(D).
\end{aligned}
\]
\end{proof}

We have derived the following results:
\begin{cor}
For the bounded case, we have the following:
\[
\begin{aligned}
\sigma(D) &= \overline{\Lambda}\\
\spp(D) &= \Lambda \\
\sep(D) &= \partial \Lambda = \overline{\Lambda} - \Lambda\\
\sepp(D) &= \Lambda - \Lambda^{\ast}\\
\se(D) &= \sep(D) \cup \sepp(D)\\
\therefore \se(D) &= (\overline{\Lambda} - \Lambda) \cup 
(\Lambda - \Lambda^{\ast})\\
&= \partial \Lambda \cup (\Lambda \setminus \Lambda^{\ast})\\
\end{aligned}
\]
\end{cor}

\begin{cor}
For the unbounded case, we have the following:
\[
\begin{aligned}
\sigma(D) &= \kf\\
\spp(D) &= \Lambda \\
\sep(D) &= \kf - \Lambda\\
\sepp(D) &= \Lambda - \Lambda^{\ast}\\
\se(D) &= \sep(D) \cup \sepp(D)\\
\therefore \se(D) &= (\kf - \Lambda) \cup 
(\Lambda - \Lambda^{\ast})\\
&= \partial \Lambda \cup (\Lambda \setminus \Lambda^{\ast})
\textrm{ where } 
\Lambda^{\ast} = \{\lambda \in \Lambda : r_{\lambda}< \infty\}
\\
\end{aligned}
\]
\end{cor}

\begin{ex}
Let 
\[
\begin{aligned}
A: \eo &\rightarrow \eo \\
u = (u_{i}) & \mapsto v = (v_{i})
\end{aligned}
\]
be defined by $v_{0} = 0$ and $v_{n+1} = u_{n}$ for all $n \in \N$.

\begin{enumerate}
\item We show that $\lambda I - A$ is injective for all $\lambda \in 
\kf$. Clearly $A$ is injective. So if $\lambda I - A$ is not 
injective then $\lambda \ne 0$. Let 
$\displaystyle \mu = \sum_{i\in \N}\mu_{i}e_{i} \in \eo, \mu_{i}\in
\kf$ be an element of $N(\lambda I - A)$. Then 
\begin{equation}\lambda u_{n}
- v_{n} = 0\quad \forall n \in \N.
\end{equation}
 Since $v_{n+1} = u_{n}$, the above equation becomes 
$ \lambda v_{n+1}- v_{n} = 0\quad \forall n \in \N$ and 
$\lambda u_{0} = v_{0} = 0 \Rightarrow  u_{0} = 0$ since 
$\lambda \ne 0$. B induction, we claim that $u_{n} = 0$. Therefore 
$N(\lambda I - A) = 0 \Rightarrow  \lambda I -A$ is injective 
$\forall k \in \kf$. 

\item Note that $||A|| = 1$. by theorem 3, if $|\lambda| > 1$, 
$\lambda I - A$ is invertible $\Rightarrow  \sigma(A) \subseteq
X = \{\lambda \in \kf \text{ such that } |\lambda| \le 1\}$. 
Let $\displaystyle u = \sum_{i\in \N}u_{i}e_{i}\in \eo$ such that
$(\lambda I - A)(u) = e_{0}$. If $\lambda = 0 \Rightarrow 0 = 1$
is impossible. If $\lambda \ne 0$, $\lambda u_{0}= 1 \Rightarrow 
u_{0} = \lambda^{-1}$. Since we have $\lambda u_{n+1} - u_{n} = 0$
$\forall n \in \N$ which we obtained from the equation 
$\lambda u_{n+1} - v_{n+1} = 0$ $\forall n\in \N$. Then for $n=0$, 
$\lambda u_{1}- u_{0} = 0\Rightarrow  u_{1} = \lambda^{-1}u_{0}
=\lambda^{-1}\lambda^{-1} = \lambda^{-2}$. By induction, we obtain 
$u_{n} = \lambda^{-n-1} \Rightarrow \lim_{n\rightarrow \infty}
|w_{n}^{1/2}||\lambda^{-n-1}|=0$. 

If $w$ is an element such that $\lim_{n\rightarrow \infty}(w_{n})
= l \ne 0$, we have that, for $|\lambda|\le 1$,  
$|w^{1/2}||\lambda^{-n-1}|$ does not go to zero. Therefore there
does not exist an element $\displaystyle u = \sum_{i \in \N}u_{i}e_{i}
\in \eo$ such that $(\lambda I -A)(u) = e_{0}$ $\forall \lambda \in \kf,
\lambda \ne 0$ and $|\lambda| \le 1$.  

Hence for $\lambda \in \kf, \lambda \ne 0, |\lambda| \le 1$, 
$\lambda I -A$ is not surjective. Since $\lambda I -A$ is injective
for all $\lambda$, $\lambda I - A$ is not a Fredholm operator of index
zero for all $\lambda \in \kf, \lambda  \ne 0, |\lambda| \le 1$. 
\end{enumerate}
\end{ex}

\end{document}